\documentclass[10pt,a4paper]{amsart}
\usepackage[font=small]{caption}

\usepackage{bbm,amsmath,amssymb,amsfonts,graphicx,color,subfig,hyperref,diagbox,booktabs,float}
\restylefloat{table}

\DeclareMathAlphabet\EuScript{U}{eus}{m}{n}
\SetMathAlphabet\EuScript{bold}{U}{eus}{b}{n}

\DeclareSymbolFont{rsfs}{U}{rsfs}{m}{n}
\DeclareSymbolFontAlphabet{\mathrsfs}{rsfs}

\theoremstyle{plain}

\newtheorem{theorem}{Theorem}
\newtheorem{lemma}[theorem]{Lemma}

\newtheorem{conjecture}{Conjecture}

\newcommand{\R}{{\mathbb{R}}}
\newcommand{\N}{{\mathbb{N}}}
\newcommand{\C}{{\mathbb{C}}}
\newcommand{\Cay}{\operatorname{Cay}}

\newcommand{\Tr}{\operatorname{Tr}}
\newcommand{\id}{1}

\newcommand{\psd}{\succeq}

\newcommand{\hatGamma}{\hat\Gamma}

\DeclareMathOperator{\U}{U}

\title{Fourier analysis on finite groups and\\ the Lov\'asz theta-number of Cayley graphs}

\author{Evan DeCorte} 
\address{P.E.B.~DeCorte, Delft Institute of Applied Mathematics, Delft
  University of Technology, P.O. Box 5031, 2600 GA Delft, The Netherlands}
\email{p.e.b.decorte@tudelft.nl}

\author{David de Laat}

\address{D.~de Laat, Delft Institute of Applied Mathematics, Delft
  University of Technology, P.O. Box 5031, 2600 GA Delft, The
  Netherlands} 
\email{mail@daviddelaat.nl}

\author{Frank Vallentin} 
\address{F.~Vallentin, Mathematisches Institut, Universit\"at zu
  K\"oln, Weyertal 86--90, 50931 K\"oln, Germany}
\email{frank.vallentin@uni-koeln.de}

\thanks{The authors were supported by Vidi grant
  639.032.917 from the Netherlands Organization for Scientific
  Research (NWO)}

\subjclass{05A05, 20C30, 90C22}

\keywords{Lov\'asz theta number, Cayley graphs of finite groups,
  independent sets, Erd\H{o}s-Ko-Rado theorems, intersecting families
  of permutations (and its $q$-analog), finite Fourier analysis}

\date{July 21, 2013}

\begin{document}

\maketitle

\markboth{P.E.B.~DeCorte, D.~de Laat, and
F.~Vallentin}{Fourier analysis on finite groups and the Lov\'asz $\vartheta$-number of Cayley graph}

\begin{abstract}
  We apply Fourier analysis on finite groups to obtain simplified
  formulations for the Lov\'asz $\vartheta$-number of a Cayley graph.
  We put these formulations to use by checking a few cases of a
  conjecture of Ellis, Friedgut, and Pilpel made in a recent article
  proving a version of the Erd\H{o}s-Ko-Rado theorem for
  $k$-intersecting families of permutations.  We also introduce a
  $q$-analog of the notion of $k$-intersecting families of
  permutations, and we verify a few cases of the corresponding Erd\H
  os-Ko-Rado assertion by computer.
\end{abstract}

\section{Introduction}

One approach to some problems in extremal combinatorics involves
estimating the independence number of a Cayley graph. A classic
example is upper bounding sizes of error-correcting codes in Abelian
groups. A recent, exciting example is provided by a version of the
Erd\H{o}s-Ko-Rado theorem for permutations proven by Ellis, Friedgut,
and Pilpel \cite{ellis11}: If $k$ is a positive integer, $n$ is
sufficiently large depending on $k$, and $\mathcal{A}$ is a largest set of
permutations on $n$ letters such that any two agree on at least $k$
letters, then $|\mathcal{A}| = (n-k)!$. This resolved a conjecture of Frankl and
Deza from \cite{FranklDeza} stated in 1977.

The Lov\'asz $\vartheta$-number, introduced in \cite{lovasz79},
provides an upper bound on the size of an independent set in a general
graph. It can be computed by solving a semidefinite program involving
$n \times n$-matrices, where $n$ is the cardinality of the vertex
set. We specialize the $\vartheta$-number to Cayley graphs and show
how the semidefinite program block-diagonalizes to a simpler one
involving smaller matrices associated to the irreducible
representations of the group. The resulting semidefinite program can
be thought of as a ``frequency domain'' formulation of the
$\vartheta$-number. Furthermore, under a sufficient condition on the
graph, our semidefinite program collapses to a linear program which
can be formulated using only knowledge of the group characters. This
condition applies, in particular, for the two examples given above. In
fact, one can interpret the arguments in \cite{ellis11} as
constructing feasible solutions to the linear program computing the
$\vartheta$-number for a particular Cayley graph on the symmetric
group.

In \cite{ellis11}, the problem of quantifying the dependence of $n$ on
$k$ is left open, but they conjecture that the conclusion of their
theorem holds when $n \geq 2k+1$. By explicit computations we verify
their conjecture for some small values of $n$ and $k$, and we identify
some values for which the $\vartheta$-number does not give a tight
enough bound to verify the conjecture, suggesting that other methods
will be required to resolve these cases.

The outline of the paper is as follows: In Section~\ref{sec:defs} we
fix notation and definitions, and recall some basic facts from finite
Fourier analysis. In Section~\ref{sec:cayley} we find several
reformulations of the Lov\'{a}sz $\vartheta$-function for Cayley
graphs by using the group structure on the vertex set. In
Section~\ref{sec:efp} we apply these results in the context of the
Ellis-Friedgut-Pilpel conjecture made in \cite{ellis11}, and
in~Section \ref{sec:qanalog} we introduce a $q$-analog of their result
as a conjecture, and perform the analogous computations.
In~Section~\ref{sec:blowup}, we show how the machinery developed in
Section~\ref{sec:cayley} could also be applied to vertex-transitive
graphs.

\section{Definitions, notation, and background in Fourier analysis}
\label{sec:defs}

All graphs will be simple and undirected. For any graph $G = (V,E)$,
the \emph{independence number} is the maximum number of pairwise
nonadjacent vertices; this maximum will be denoted $\alpha(G)$.

Suppose $\Gamma$ is a finite group. A subset $X \subseteq \Gamma$ will
be called a \emph{connection set} if the unit element $e$ of $\Gamma$
does not belong to $X$, and if $X$ is inverse-closed; that is $x^{-1}
\in X$ whenever $x \in X$.  For any connection set $X \subseteq
\Gamma$, the \emph{Cayley graph} $\Cay(\Gamma, X)$ is the graph with
vertex set $\Gamma$, where two vertices $x$ and $y$ are adjacent if
and only if $y^{-1}x \in X$. The defining conditions of a connection
set imply that $\Cay(\Gamma,X)$ is an undirected graph without
self-loops. Notice that we do not require $X$ to generate $\Gamma$;
therefore $\Cay(\Gamma, X)$ need not be connected.

In the following we recall some basic facts from representation theory
of finite groups. For a good reference, see for instance
Terras~\cite{terras99}.  A (finite-dimensional) \emph{unitary
  representation} of $\Gamma$ is a group homomorphism $\pi \colon
\Gamma \to \U(d_\pi)$ where $\U(d_\pi)$ is the group of unitary $d_\pi
\times d_\pi$ matrices. The number~$d_\pi$ is called the \emph{degree}
of~$\pi$. The \emph{character} of $\pi$ is defined as
$\chi_\pi(\gamma) = \Tr(\pi(\gamma))$, where $\Tr$ denotes trace. A
subspace $M$ of $\C^{d_\pi}$ is \emph{$\pi$-invariant} if $\pi(\gamma)
m \in M$ for all $\gamma \in \Gamma$ and $m \in M$. The unitary
representation $\pi$ is said to be \emph{irreducible} if $\{0\}$ and
$\C^{d_\pi}$ are the only $\pi$-invariant subspaces of
$\C^{d_\pi}$. Two unitary representations $\pi$ and $\pi'$ are
(unitarily) \emph{equivalent} if there is a unitary matrix $T$ such
that $T\pi(\gamma) = \pi'(\gamma)T$ for all $\gamma \in \Gamma$.

Given two inequivalent irreducible unitary representations $\pi$ and
$\pi'$, the \emph{Schur orthogonality relations} give us the following
two facts:
\begin{enumerate}
\item $\sum_{\gamma \in \Gamma} \pi_{ij}(\gamma)
  \overline{\pi'_{lk}(\gamma)} = 0$, where $\pi_{ij}(\gamma)$ is the
  $ij$-entry of the matrix $\pi(\gamma)$, and $\pi'_{lk}(\gamma)$ is
  defined analogously;
\item $\sum_{\gamma \in \Gamma} \pi_{ij}(\gamma)
  \overline{\pi_{lk}(\gamma)} = \frac{|\Gamma|}{d_\pi} \delta_{il}
  \delta_{jk}$, where $\delta$ is the Kronecker delta.
\end{enumerate}

These relations are implied by \emph{Schur's lemma}, which says that
if $\pi$ and $\pi'$ are irreducible unitary representations, and if
$T$ is a matrix for which $T \pi(\gamma) = \pi'(\gamma) T$ for all
$\gamma \in \Gamma$, then $T$ is either invertible or zero; if $\pi =
\pi'$, then $T$ is a scalar multiple of the identity matrix.

We fix a set of mutually inequivalent irreducible unitary
representations of~$\Gamma$, so that each unitary equivalence class
has a representative; call this set $\hatGamma$. This allows us to
define the \emph{Fourier transform} of a function $f \colon \Gamma \to
\C$:
\[
\hat f(\pi) = \sum_{\gamma \in \Gamma} f(\gamma) \pi(\gamma), 
\]
where $\hat f(\pi)$ is a complex $d_\pi \times d_\pi$ matrix.  The
\emph{Fourier inversion formula} says we can recover $f$ from its
Fourier transform:
\[
f(\gamma) = \frac{1}{|\Gamma|} \sum_{\pi \in \hatGamma} d_\pi \langle \hat f(\pi), \pi(\gamma) \rangle.
\]
The inner product used here is the trace inner product, defined as
$\langle A, B \rangle = \Tr(B^*A)$ for square complex matrices $A$ and
$B$ of the same dimension, where $B^*$ denotes the conjugate-transpose
of $B$.

The \emph{convolution} of two functions $f \colon \Gamma \to \C$ and
$g \colon \Gamma \to \C$ is defined by
\[
f * g(\gamma) = \sum_{\beta \in \Gamma} f(\beta) g(\beta^{-1} \gamma),
\]
and the \emph{involution} of $f$ is defined as $f^*(\gamma) =
\overline{f(\gamma^{-1})}$.  It is a fact that $\widehat{f*g}(\pi) =
\hat f(\pi) \hat g(\pi)$, and that $\widehat{f^*}(\pi) = \hat
f(\pi)^*$.

A function $f \colon \Gamma \to \C$ is of \emph{positive type} if
\[
\sum_{\gamma \in \Gamma} g*g^*(\gamma) f(\gamma) \geq 0
\]
for all functions $g \colon \Gamma \to \C$; that is, the sum is a
nonnegative real number.  We denote by $\mathcal{P}(\Gamma)$ the set
of functions on $\Gamma$ of positive type.  Notice that $f \in
\mathcal{P}(\Gamma)$ if and only if $\bar{f} \in \mathcal{P}(\Gamma)$,
where $\bar{f}$ is the pointwise complex-conjugate of $f$. One fact
that will be needed later is that $f(\gamma^{-1}) =
\overline{f(\gamma)}$ for all $\gamma \in \Gamma$ when $f$ is of
positive type. For a proof of this fact and more information on
functions of positive type, see Folland \cite[Chapter 3.3]{folland95}.

For vectors $u,v \in \C^n$, we use $\langle u, v \rangle$ to denote
the usual inner product of $u$ and $v$.  An $n \times n$ matrix $A$
with entries from $\C$ will be called \emph{positive semidefinite} if
$\langle Av, v \rangle$ is a nonnegative real number for all $v \in
\C^n$. Using the polarization identity, it is possible to prove that
every positive semidefinite matrix is Hermitian.  For each finite set
$V$, the set of positive semidefinite matrices with rows and columns
indexed on $V$ will be denoted $\mathcal{S}_{\psd 0}^V$. When $V =
\{1, \ldots, n\}$, we will use the notation $\mathcal{S}_{\psd 0}^n$
instead.  It is a fact that $A \in \mathcal{S}^n_{\psd 0}$ if and only
if $\langle A, B \rangle \geq 0$ for all $B \in \mathcal{S}^n_{\psd
  0}$; this fact is known as the \emph{self-duality} of
$\mathcal{S}^n_{\psd 0}$.

The following theorem is an application of self-duality, as well as
\emph{Parseval's identity}, which says that
\[
 \sum_{\gamma \in \Gamma} f(\gamma) \overline{g(\gamma)} = \frac{1}{|\Gamma|} \sum_{\pi \in \hatGamma} d_\pi \langle \hat f(\pi), \hat g(\pi) \rangle
\]
for all functions $f$ and $g$ on $\Gamma$:

\begin{theorem}[Bochner's theorem for finite groups]
\label{thm:bochner}
Suppose $\Gamma$ is a finite group and let $f \colon \Gamma \to
\C$. Then $f$ is of positive type if and only if $\hat f(\pi)$ is
positive semidefinite for each $\pi \in \hatGamma$.
\end{theorem}
\begin{proof}
  For any two complex-valued functions $f$ and $g$ on $\Gamma$, we
  have
  \begin{equation}\label{parsevaleq}
   \sum_{\gamma \in \Gamma} g*g^*(\gamma) \overline{f(\gamma)} =
   \frac{1}{|\Gamma|} \sum_{\pi \in \hatGamma} d_\pi \langle \widehat{g*g^*}(\pi), \hat f(\pi) \rangle =
   \frac{1}{|\Gamma|} \sum_{\pi \in \hatGamma} d_\pi \langle \hat g(\pi) \hat g(\pi)^*, \hat f(\pi) \rangle.
  \end{equation}
  The matrices $\hat g(\pi) \hat g(\pi)^*$ are always positive
  semidefinite, so \eqref{parsevaleq} is nonnegative if all the
  matrices $\hat f(\pi)$ are positive semidefinite. This gives one
  direction.
  
  For the other direction, suppose $f \colon \Gamma \to \C$ is of
  positive type, and fix $\pi \in \hatGamma$. Now let $A \in
  \mathcal{S}^{d_\pi}_{\psd 0}$ be arbitrary, and let $A = BB^*$ be
  the Cholesky decomposition. Define $g \colon \Gamma \to \C$ by
  $g(\gamma) = d_\pi / |\Gamma| \langle B, \pi(\gamma) \rangle$. By
  the Schur orthogonality relations (or uniqueness of Fourier
  coefficients), we have $\hat g(\pi) = B$ and $\hat g(\pi') = 0$ when
  $\pi'$ and $\pi$ are inequivalent, whence
  \[
    \hat g(\pi) \hat g(\pi)^* = BB^* = A \quad \text{and} \quad \hat g(\pi') \hat g(\pi')^* = 0.
  \]
  Now \eqref{parsevaleq}, which is nonnegative by hypotheses, is equal
  to $d_\pi/|\Gamma| \langle A, \hat f(\pi) \rangle$.  Since $\pi$ and
  $A$ were arbitrary, we conclude that $\langle A, \hat f(\pi) \rangle
  \geq 0$ for every $\pi$ and every $A \in \mathcal{S}^{d_\pi}_{\psd
    0}$.  Self-duality of $\mathcal{S}^{d_\pi}_{\psd 0}$ now implies
  $\hat f(\pi) \in \mathcal{S}^{d_\pi}_{\psd 0}$ for each $\pi \in
  \hatGamma$.
\end{proof}

\section{The $\vartheta$-number of a Cayley graph}
\label{sec:cayley}

Let $G = (V,E)$ be a finite graph. In \cite{lovasz79}, the Lov\'{a}sz
$\vartheta$-number $\vartheta(G)$ of $G$ is defined and a number of
equivalent formulations are given. The formulation of $\vartheta(G)$
which will be most important for us is:
\begin{align}\tag{A}\label{primal}
\vartheta(G) = \max \Big\{ \sum_{u, v \in V} A(u,v) : \; & A \in \mathcal{S}_{\psd 0}^V \text{ real-valued},\\[-3ex]
& \Tr(A) = 1, \; A(u, v) = 0 \text{ for } \{u,v\} \in E \Big\}.\nonumber
\end{align}
When $G$ is the Cayley graph $\Cay(\Gamma,X)$, the optimization over
matrices in \eqref{primal} can be replaced with optimization over
functions on $\Gamma$, as we proceed to show.

\begin{theorem}\label{ptFunctions}
Suppose $G=\Cay(\Gamma,X)$. Then 
\begin{align}
  \tag{B}\label{pt}
  \vartheta(G) = \max \Big\{ \sum_{\gamma \in \Gamma} f(\gamma) : \; & f \in \mathcal{P}(\Gamma) \text{ real-valued},\\[-3ex]
  & f(e) = 1, \; f(x) = 0 \text{ for } x \in X \Big\}.\nonumber
\end{align}
\end{theorem}

Before we prove Theorem \ref{ptFunctions}, we require a lemma:

\begin{lemma}\label{tensorConv}
  Suppose $A \colon \Gamma \times \Gamma \to \C$ is a Hermitian matrix
  satisfying $A(\gamma,e) = A(\gamma \beta, \beta)$ for all
  $\gamma,\beta \in \Gamma$. Define $f \colon \Gamma \to \C$ by
  $f(\gamma) = A(\gamma,e)$. Then for any function $g \colon \Gamma
  \to \C$ we have
\[
\sum_{\gamma \in \Gamma} g*g^*(\gamma) f(\gamma) = \sum_{\gamma, \gamma' \in \Gamma} g(\gamma) \overline{g(\gamma')} A(\gamma, \gamma').
\]
\end{lemma}
\begin{proof}
This follows from a straightforward computation.
\end{proof}

\begin{proof}[Proof of Theorem \ref{ptFunctions}]
  For one direction, let $A$ be a feasible solution for
  \eqref{primal}. Define $\bar A \colon \Gamma \times \Gamma \to \R$
  entrywise by
\[
\bar A(\gamma, \gamma') = \frac{1}{|\Gamma|} \sum_{\beta \in \Gamma}
A(\gamma \beta, \gamma' \beta).
\]
Being the average of matrices similar to $A$ (via permutation
matrices), the matrix $\bar A$ is positive semidefinite, and one now
easily checks that $\bar A$ is again a feasible solution for
\eqref{primal} having the same objective value as $A$. Moreover, we
have $\bar A(\gamma,e) = \bar A(\gamma \beta, \beta)$ for all
$\gamma,\beta \in \Gamma$.

Now define $f \colon \Gamma \to \R$ by $f(\gamma) = |\Gamma| \bar
A(\gamma,e)$. Then $\bar A$ and $f/|\Gamma|$ satisfy the hypotheses of
Lemma~\ref{tensorConv}, so
\[
\sum_{\gamma \in \Gamma} g*g^*(\gamma) f(\gamma)
= |\Gamma| \sum_{\gamma, \gamma' \in \Gamma} g(\gamma) \overline{g(\gamma')}
\bar A(\gamma, \gamma'),
\]
and since $\bar A$ is positive semidefinite, it follows that the
function $f$ is of positive type.  It is easily checked that the other
constraints of \eqref{pt} are satisfied by $f$, and moreover that the
objective values are equal:
\[
\sum_{\gamma \in \Gamma} f(\gamma) = |\Gamma| \sum_{\gamma \in \Gamma} \bar A(\gamma,e)
= \sum_{\gamma,\gamma' \in \Gamma} \bar A(\gamma,\gamma')
= \sum_{\gamma,\gamma' \in \Gamma} A(\gamma,\gamma').
\]

For the other direction, we begin with a feasible solution $f \colon
\Gamma \to \R$ to \eqref{pt}, and we define $A \colon \Gamma \times
\Gamma \to \R$ by $A(\beta,\gamma) = \frac{1}{|\Gamma|} f(\beta
\gamma^{-1})$.  Then $A$ is a feasible solution to \eqref{primal} by
Lemma \ref{tensorConv}, and its objective value is $\sum_{\gamma \in
  \Gamma} f(\gamma)$.
\end{proof}

Using Theorem \ref{thm:bochner}, we can also give a (complex)
semidefinite programming formulation of \eqref{pt} using block
matrices.
\begin{theorem}\label{comp}
Suppose $G=\Cay(\Gamma,X)$. Then
\begin{align}
\tag{C}\label{compFormulation}
\vartheta(G) = \max \Big\{ A_\id : \; & A_\pi \in \mathcal{S}_{\succeq 0}^{d_\pi} \text{ for each } \pi \in \hatGamma, \\[-0.3ex]
&\sum_{\pi \in \hatGamma} d_\pi \Tr(A_\pi) = |\Gamma|, \; \sum_{\pi \in \hatGamma} d_\pi \langle A_\pi, \pi(x) \rangle = 0 \text{ for } x \in X\Big\}\nonumber,
\end{align}
where $1 \in \hatGamma$ denotes the trivial representation.
\end{theorem}
\begin{proof}
  If $f \colon \Gamma \to \R$ is any feasible solution to \eqref{pt},
  set $A_\pi = \hat{f}(\pi)$ for each $\pi \in \hat \Gamma$.  By
  Theorem \ref{thm:bochner}, the matrices $A_\pi$ are positive
  semidefinite.  Moreover, one easily checks using the Fourier
  inversion formula that the other constraints of
  \eqref{compFormulation} are satisfied by $\{ A_\pi : \pi \in
  \hatGamma\}$, and that the objective values are equal: $ A_1 =
  \sum_{\gamma \in \Gamma} f(\gamma).  $

  For the other direction, let $\{ A_\pi : \pi \in \hatGamma\}$ be a
  feasible solution for \eqref{compFormulation} and define $g \colon
  \Gamma \to \C$ by
\[
g(\gamma) = \frac{1}{|\Gamma|} \sum_{\pi \in \hatGamma} d_\pi \langle A_\pi, \pi(\gamma) \rangle \quad \text{for all} \quad \gamma \in \Gamma.
\]
Then $g$ is of positive type by Theorem \ref{thm:bochner}.  Now define
$f(\gamma) = \frac{1}{2}(g(\gamma) + g(\gamma^{-1}))$ for all $\gamma
\in \Gamma$.  Then $f$ is real-valued, and that $f$ satisfies all the
other constraints of \eqref{pt} is easily checked using the fact that
$X$ is inverse-closed. Moreover
\[
\sum_{\gamma \in \Gamma} f(\gamma)
= \frac{1}{|\Gamma|} \sum_{\gamma \in \Gamma}
\sum_{\pi \in \hatGamma} d_\pi \langle A_\pi, \pi(\gamma)\rangle
= A_\id
\]
by the Schur orthogonality relations.
\end{proof}

When $\Gamma$ is an Abelian group, then all its irreducible
representation are one-dimensional. Therefore, the semidefinite
program~\eqref{compFormulation} is just a linear program. More
generally, \eqref{compFormulation} is equivalent to a linear program
whenever the connection set of the Cayley graph $\Cay(\Gamma,X)$ is
closed under conjugation; that is, $\gamma x \gamma^{-1} \in X$ for
all $x \in X$ and $\gamma \in \Gamma$. This is the content of the next
theorem.

\begin{theorem}\label{lp}
  Let $G$ be the Cayley graph $\Cay(\Gamma, X)$ and suppose that the
  connection set $X$ is closed under conjugation. Then
\begin{align}
\tag{D}\label{lpFormulation}
\vartheta(G) = \max \Big\{ a_\id : \; & a_\pi \geq 0 \text{
  for each } \pi \in \hatGamma, \\[-0.4ex]
&  \sum_{\pi \in \hatGamma} d_\pi^2 a_\pi = |\Gamma|, \;
\sum _{\pi \in \hatGamma} d_\pi a_\pi \chi_\pi(x) = 0 \text{ for } x \in X \Big\}.\nonumber
\end{align}
\end{theorem}
\begin{proof}
  We prove the equivalence of \eqref{compFormulation} and
  \eqref{lpFormulation}.  Let $\{ A_\pi : \pi \in \hatGamma \}$ be a
  feasible solution for \eqref{compFormulation}, and for each $\pi$
  let
\[
\bar A_\pi = \frac{1}{| \Gamma |} \sum_{\gamma \in \Gamma} \pi(\gamma) A_\pi \pi(\gamma)^*.
\]
Then $\{ \bar A_\pi : \pi \in \hatGamma \}$ is again a solution to \eqref{compFormulation}: If $x \in X$, then
\begin{align*}
\sum_{\pi \in \hatGamma} d_\pi \langle \bar A_\pi, \pi(x) \rangle 
&= \frac{1}{|\Gamma|} \sum_{\pi \in \hatGamma} d_\pi \sum_{\gamma \in \Gamma} \langle \pi(\gamma) A_\pi \pi(\gamma)^*, \pi(x) \rangle \\
&= \frac{1}{|\Gamma|} \sum_{\pi \in \hatGamma} d_\pi \sum_{\gamma \in
  \Gamma} \langle \pi(\gamma) A_\pi, \pi(x \gamma) \rangle.
\end{align*}
Since $X$ is closed under conjugation there is a $y \in X$ so that
$x\gamma = \gamma y$ holds.  Hence, the sum above equals
\[
\frac{1}{|\Gamma|} \sum_{\pi \in \hatGamma} d_\pi \sum_{\gamma \in \Gamma} \langle \pi(\gamma) A_\pi, \pi(\gamma y) \rangle 
= \frac{1}{|\Gamma|} \sum_{\pi \in \hatGamma} d_\pi \sum_{\gamma \in \Gamma} \langle A_\pi, \pi(y) \rangle 
= \sum_{\pi \in \hatGamma} d_\pi \langle A_\pi, \pi(y) \rangle = 0.
\]
Moreover, since $\pi(\gamma) A_{\pi} \pi(\gamma)^*$ is similar to
$A_\pi$ for each $\gamma \in \Gamma$, the matrix $\bar A_\pi$ is
positive semidefinite for each $\pi \in \hatGamma$ and $\sum_{\pi \in
  \hatGamma} d_\pi \Tr(\bar A_\pi) = |\Gamma|$.

We have constructed $\bar A_\pi$ so that $\bar A_\pi \pi(\gamma) =
\pi(\gamma) \bar A_\pi$ for all $\gamma \in \Gamma$. Schur's lemma
then implies that $\bar A_\pi$ is equal to $a_\pi I_{d_\pi}$ for some
scalar $a_\pi$ and since $\bar A_\pi$ is positive semidefinite this
scalar is nonnegative. We have $d_\pi a_\pi = \Tr(\bar A_\pi)$ as well
as
\[
\langle \bar A_\pi, \pi(\gamma) \rangle = a_{\pi} \chi_\pi(\gamma) \quad \text{for all} \quad \gamma \in \Gamma,
\]
so $\{ a_\pi : \pi \in \hatGamma \}$ is a feasible solution to \eqref{lpFormulation}
having objective value $a_\id = A_\id$.

For the other direction, we take a feasible solution $\{ a_\pi : \pi \in \hatGamma \}$
to \eqref{lpFormulation}, and for each $\pi \in
\hatGamma$, we set $A_\pi = a_\pi I_{d_\pi}$. This is a
feasible solution to \eqref{compFormulation} with objective value
$A_\id = a_\id$.
\end{proof}

Denote the constraint $\sum _{\pi \in \hatGamma} d_\pi a_\pi
\chi_\pi(x) = 0$ by $C_x$ ($x \in X$). For computational purposes, the
following simplifications can be applied to \eqref{lpFormulation}:
First, only one of the constraints $\{ C_x, C_{x^{-1}} \}$ is needed.
Second, since the characters $\chi_\pi$ are constant on conjugacy
classes, it suffices to keep only the constraints $C_x$, with one $x$
per conjugacy class.

\section{First application: $k$-intersecting permutations}
\label{sec:efp}

In this section we apply Theorem~\ref{lp} to the problem of
$k$-intersecting permutations as discussed in the introduction.

Let $S_n$ be the symmetric group on $n$ letters. A family $\mathcal A
\subseteq S_n$ is said to be \emph{$k$-intersecting} if any two
permutations in $\mathcal A$ agree on at least $k$ elements. That is,
a $k$-intersecting family of $S_n$ is an independent set in the graph
$\Cay(S_n, X_{n,k})$, where
\[
X_{n,k} = \{\sigma
\in S_n : \sigma \text{ has strictly less than $k$ fixed points}\}.
\]
The set $X_{n,k}$ is closed under conjugation so Theorem~\ref{lp}
applies. One can interpret the method of Ellis, Friedgut, and Pilpel
in~\cite{ellis11} as constructing an explicit family of feasible
solutions to the linear programs which turns out to be optimal for
given $k$ and $n$ sufficiently large.

Conjecture 2 of \cite{ellis11} implies that a largest $k$-intersecting
family in $S_n$ has size
\[
\max_{0 \leq i \leq (n-k)/2} \big|\{\sigma \in S_n \colon \sigma \text{ has at least $k+i$ fixed points in } \{1, \ldots, k+2i\}\}\big|,
\]
which in particular means that the maximum size is $(n-k)!$ for $n
\geq 2k+1$.  We solved the linear program \eqref{lpFormulation} for
small values of $n$ and $k$ with the help of a computer. In
Table~\ref{comp1} the $(n, k)$-th entry is marked when the
$\vartheta$-number gives the conjectured maximum. To evaluate the
characters of the symmetric group we used \texttt{gap}~\cite{gap} and
to solve the linear programs we used \texttt{lrs}~\cite{lrs}. Since
both software packages only use rational arithmetic our computations
are rigorous.

\begin{table}[H]
\centering
\scriptsize
\begin{tabular}{cccccccccccccccc}
\toprule
\diagbox{k}{n} &$1$&$2$&$3$&$4$&$5$&$6$&$7$&$8$&$9$&$10$&$11$&$12$&$13$&$14$&$15$\\
\midrule
$1$ & $\checkmark$ & $\checkmark$ & $\checkmark$ & $\checkmark$ & $\checkmark$ & $\checkmark$ & $\checkmark$ & $\checkmark$ & $\checkmark$ & $\checkmark$ & $\checkmark$ & $\checkmark$ & $\checkmark$ & $\checkmark$ & $\checkmark$\\
\midrule
$2$ &  & $\checkmark$ & $\checkmark$ & $\checkmark$ & $\checkmark$ & $\checkmark$ & $\checkmark$ & $\checkmark$ & $\checkmark$ & $\checkmark$ & $\checkmark$ & $\checkmark$ & $\checkmark$ & $\checkmark$ & $\checkmark$\\
\midrule
$3$ &  &  & $\checkmark$ & $\checkmark$ & $\checkmark$ & $\checkmark$ &  & $\checkmark$ & $\checkmark$ & $\checkmark$ & $\checkmark$ & $\checkmark$ & $\checkmark$ & $\checkmark$ & $\checkmark$\\
\midrule
$4$ &  &  &  & $\checkmark$ & $\checkmark$ & $\checkmark$ &  &  &  &  & $\checkmark$ & $\checkmark$ & $\checkmark$ & $\checkmark$ & $\checkmark$\\
\midrule
$5$ &  &  &  &  & $\checkmark$ & $\checkmark$ & $\checkmark$ &  &  &  &  & $\checkmark$ & $\checkmark$ & $\checkmark$ & $\checkmark$\\
\midrule
$6$ &  &  &  &  &  & $\checkmark$ & $\checkmark$ & $\checkmark$ &  &  &  &  &  &  & $\checkmark$\\
\midrule
$7$ &  &  &  &  &  &  & $\checkmark$ & $\checkmark$ & $\checkmark$ &  & $\checkmark$ &  &  &  & \\
\midrule
$8$ &  &  &  &  &  &  &  & $\checkmark$ & $\checkmark$ & $\checkmark$ &  &  &  &  & \\
\midrule
$9$ &  &  &  &  &  &  &  &  & $\checkmark$ & $\checkmark$ & $\checkmark$ & $\checkmark$ &  &  & \\
\midrule
$10$ &  &  &  &  &  &  &  &  &  & $\checkmark$ & $\checkmark$ & $\checkmark$ &  &  & \\
\midrule
$11$ &  &  &  &  &  &  &  &  &  &  & $\checkmark$ & $\checkmark$ & $\checkmark$ &  & \\
\midrule
$12$ &  &  &  &  &  &  &  &  &  &  &  & $\checkmark$ & $\checkmark$ & $\checkmark$ & \\
\midrule
$13$ &  &  &  &  &  &  &  &  &  &  &  &  & $\checkmark$ & $\checkmark$ & $\checkmark$\\
\midrule
$14$ &  &  &  &  &  &  &  &  &  &  &  &  &  & $\checkmark$ & $\checkmark$\\
\midrule
$15$ &  &  &  &  &  &  &  &  &  &  &  &  &  &  & $\checkmark$\\
\bottomrule
\end{tabular}
\vspace{2ex}
\caption{\label{comp1}Computation of $\vartheta(\Cay(S_n, X_{n,k}))$}
\vspace{-5ex}
\end{table}

\section{Second application: $k$-intersecting invertible matrices}
\label{sec:qanalog}

Here we consider a $q$-analog of the previous application. Let $\Gamma
= \mathrm{GL}(n, \mathbb{F}_q)$ be the group of invertible $n \times
n$-matrices over the finite field with $q$ elements, where $q$ is a
prime power. We say that two matrices $A$ and $B$ in $\mathrm{GL}(n,
\mathbb{F}_q)$ \emph{$k$-intersect} if there is a $k$-dimensional
subspace $H$ of $\mathbb{F}_q^n$ such that $Ax = Bx$ for all $x \in
H$. Given a natural number $k$, let
\[
X_{q,n,k} = \{ A \in \mathrm{GL}(n, \mathbb{F}_q) : \mathrm{rank}(A - I) > n-k \}
\]
and consider the Cayley graph $G_{q,n,k} = \Cay(\Gamma, X_{q,n,k})$.
Independent sets in this graph correspond to $k$-intersecting families
of invertible matrices.

The independence number of $G_{q,n,1}$ was recently calculated by Guo
and Wang in \cite{GuoWang2011} (not by computing $\vartheta(G_{q,n,1})$). 

For any $q$ and $n$, one clearly
obtains a lower bound by choosing a nonzero vector $x \in
\mathbb{F}_q^n$ and considering the set $\mathcal{A}$ of all matrices $A \in
\mathrm{GL}(n, \mathbb{F}_q)$ such that $Ax = x$. One has $|\mathcal{A}| =
\prod_{i=1}^{n-1} (q^n-q^i)$ by the orbit-stabilizer theorem, and for
small values of $n$ and $q$ we found numerically that
$\vartheta(G_{q,n,1})$ equals this lower bound.  Since $X_{q,n,k}$ is
closed under conjugation, $\vartheta(G_{q,n,k})$ can be
computed by solving the linear program \eqref{lpFormulation}.

\begin{conjecture}
  One has $\vartheta(G_{q,n,1}) = \alpha(G_{q,n,1}) =
  \prod_{i=1}^{n-1} (q^n-q^i)$ for all values of $n$ and $q$.
\end{conjecture}

For $k > 1$, we can construct independent sets in a similar way as
above: Choose linearly independent vectors $x_1, \dots, x_k \in
\mathbb{F}_q^n$ and let $\mathcal{A}$ be the set of all matrices $A \in
\mathrm{GL}(n, \mathbb{F}_q)$ such that $Ax_i = x_i$ for $1 \leq i
\leq k$.  Then $|\mathcal{A}| = \prod_{i=k}^{n-1} (q^n - q^i)$.  By computing
the $\vartheta$-number for small values of $n$ and $q$ (see Table~\ref{comp2}) we
have evidence that a version of the Erd\H os-Ko-Rado theorem might
also be true in this setting.

\begin{conjecture}
  We conjecture that for each $q, k \in \N$, there exists $n_0 =
  n_0(q,k) \in \N$ such that $\vartheta(G_{q,n,k}) = \alpha(G_{q,n,k})
  = \prod_{i=k}^{n-1} (q^n - q^i)$ for all $n \geq
  n_0$. 
\end{conjecture}

The computations in Table~\ref{comp2} have been performed with
\texttt{magma}~\cite{magma} and \texttt{lpsolve}~\cite{lpsolve}. As
the computation of the characters of $\mathrm{GL}(n, \mathbb{F}_q)$
involve irrational numbers we cannot solve the linear programs with
rational arithmetic only. So these computations cannot be considered as
rigorous mathematical proofs. Nevertheless we are certain that
we placed checkmarks where the exact computation of $\vartheta(G_{q,n,k})$
would give an upper bound which is equal to the corresponding lower bound.

\begin{table}[H]
\centering
\scriptsize
\begin{tabular}{cccccccccccccc}
\toprule
& \multicolumn{6}{c}{$q = 2$} & \multicolumn{4}{c}{$q = 3$} & \multicolumn{3}{c}{$q = 4$}\\ 
\cmidrule(lr){2-7} \cmidrule(lr){8-11} \cmidrule(lr){12-14}
\diagbox{k}{n} & $1$ & $2$ & $3$ & $4$ & $5$ & $6$ & $1$ & $2$ & $3$ & $4$ & $1$ & $2$ & $3$\\
\midrule
$1$ & $\checkmark$ & $\checkmark$ & $\checkmark$ & $\checkmark$ & $\checkmark$ & $\checkmark$ &  $\checkmark$ & $\checkmark$ & $\checkmark$ & $\checkmark$ & $\checkmark$ & $\checkmark$ & $\checkmark$\\ 
\midrule
$2$ & & $\checkmark$ & $\checkmark$ & & $\checkmark$ & $\checkmark$ & & $\checkmark$ & $\checkmark$ & $\checkmark$ & & $\checkmark$ & $\checkmark$\\
\midrule
$3$ & & & $\checkmark$ & $\checkmark$ & & $\checkmark$ & & & $\checkmark$ & $\checkmark$ & & & $\checkmark$\\
\midrule
$4$ & & & & $\checkmark$ & $\checkmark$ & & & & & $\checkmark$ & & &\\
\midrule
$5$ & & & & & $\checkmark$ & $\checkmark$  & & & & & & &\\
\midrule
$6$ & & & & & & $\checkmark$ &  & & & & & &\\
\bottomrule
\end{tabular}
\vspace{2ex}
\caption{\label{comp2}Computation of $\vartheta(\Cay(\Gamma, X_{q,n,k}))$}
\vspace{-3ex}
\end{table}

\section{Blowing up vertex transitive graphs}
\label{sec:blowup}

The final theorem in this note shows that for the purposes of
estimating the independence number of a graph, the theory presented in
the preceding sections can be applied not just to Cayley graphs, but
also to vertex-transitive graphs.

\begin{theorem}
\label{thm:blowup}
  Let $G = (V,E)$ be a graph and let $\Gamma$ be a group of
  automorphisms of $G$. Suppose $\Gamma$ acts transitively on
  $V$. Then there exists a connection set $X \subseteq \Gamma$ such that
 \[
\alpha(G) = 
\tfrac{|V|}{|\Gamma|} \alpha(\Cay(\Gamma,X)).
\]
\end{theorem}
\begin{proof}
 Pick a vertex $x_0 \in V$ and define
 \[
 X = \{ \gamma \in \Gamma : \{x_0, \gamma \cdot x_0\} \in E \}.
\]
Then for $\beta,\gamma \in \Gamma$, one has an edge $\{\beta,
\gamma\}$ in the Cayley graph $\Cay(\Gamma,X)$ if and only if
\[
\gamma^{-1} \beta \in X \iff \{x_0, \gamma^{-1} \beta \cdot x_0\} \in
E \iff \{\gamma \cdot x_0, \beta \cdot x_0\} \in E.
\] 
Now notice that by the orbit-stabilizer theorem, one has
\[
|\{\gamma \in \Gamma : \gamma \cdot x  = x\}| = \frac{|\Gamma|}{|V|}
\quad \text{for all $x \in V$},
\]
and the theorem follows immediately.
\end{proof}

Going from $G$ to the Cayley graph $\Cay(\Gamma,X)$ is accomplished
using the following procedure: First choose a vertex $x_0 \in V$
arbitrarily, and let $H$ be the stabilizer subgroup of $x_0$ in
$\Gamma$. Each vertex $x \in V$ is then replaced with an empty graph
on the left coset of $H$ in $\Gamma$ consisting of all those $\gamma
\in \Gamma$ such that $\gamma \cdot x_0 = x$.  In other words, the
vertex set $V$ is regarded as a $\Gamma$-homogeneous space, and each
vertex is ``blown up'' to an independent set of size $|\Gamma| / |V|$
by replacing it with its inverse image under the projection
map.

\end{document}